\documentclass{article}
\usepackage[utf8]{inputenc}
\usepackage{amsmath,stackengine,mathtools}
\usepackage{fullpage}
\usepackage[ruled,vlined]{algorithm2e}
\usepackage[square,numbers]{natbib}
\usepackage{xspace}

\usepackage{macro}
\usepackage{authblk}
\usepackage{hyperref}
\usepackage{cleveref}

\title{A Positivstellensatz for Conditional SAGE Signomials\thanks{Revised December 2, 2026. This version corrects errors in the proofs of Theorem~\ref{thm:sage_p_inequality} and Theorem~\ref{thm:infinite_to_finite}, the toy example, and the convergence claim in Section~4, as well as typographical errors. The main result and proof strategy are unchanged. \\}}

\author[1]{Allen H. Wang \thanks{Email: allen.wang@uwaterloo.ca. Part of the work was done during an internship at Borealis AI.}}
\author[1]{Priyank Jaini \thanks{Email: pjaini@uwaterloo.ca.}}
\author[1]{Yaoliang Yu \thanks{Email: yaoliang.yu@uwaterloo.ca}}
\author[1,2]{Pascal Poupart \thanks{Email: ppoupart@uwaterloo.ca.}}
\affil[1]{University of Waterloo, Vector Institute.}
\affil[2]{Borealis AI}
\date{}
\begin{document}

\maketitle
\begin{abstract}
The Conditional SAGE certificates were introduced as sufficient conditions for constrained signomial nonnegativity over a convex set. In this article, we show that the Conditional SAGE certificate has a hierarchy that is \textit{complete}. That is, for any signomial $f(\xv) = \sum_{j=1}^{\ell}c_j \exp(\Av_j\xv)$ defined by rational exponents that is positive over a compact convex set $\mathcal{X}$, there exist $p \in \mathbb{Z}_+$ and a specific positive definite function $w(\xv)$ such that $w(\xv)^p f(\xv)$ may be verified by the Conditional SAGE certificate. The completeness result is analogous to Positivstellensatz results from algebraic geometry, which guarantee representations of positive polynomials with sum-of-squares polynomials. The result gives rise to a convergent hierarchy of lower bounds for a constrained signomial optimization problem over an \textit{arbitrary} compact convex set that is computable via the Conditional SAGE certificate.   
\end{abstract}

\section{Introduction}
A \textit{signomial} is a function of the form $f(\xv) = \sum_{j=1}^{\ell} c_j \exp(\Av_j \xv)$, where $\cv \in \mathbb{R}^\ell$ and $\Av_j \in \mathbb{R}^n$ for $j = 1 \ldots \ell$ are fixed. Optimization of such functions subject to signomial inequalities and equalities is called signomial programming (SP). Although computationally difficult, SPs have a wide range of applications in chemical engineering \cite{Rountree82}, aeronautics \cite{York2018b-sig-aero}, circuit design \cite{Jabr07}, communications network optimization \cite{Chiang09}, and machine learning \cite{Zhao16}. \\\\
Signomials may be thought of as a generalization of polynomials. By a change of variable $y_i = \exp x_i$, one obtains the expression $p(\yv) = \sum_{j=1}^{\ell} c_j \prod_{i=1}^n y_i^{\alpha_{ij}}$; in polynomials, the exponents are restricted to be integers. Algorithms for polynomial optimization, with their wide applications, have been well studied. These algorithms are based on sum-of-squares (SOS) certificates of polynomial positivity described by Positivstellensatz results from algebraic geometry and are computationally tied to semidefinite programming (SDP) \cite{Marshall08, Parrilo00}. More recently, the concept of sums of nonnegative circuit polynomials (SONC) has been proposed as a new sufficient condition for positivity suited for sparse polynomials, which may be used to design efficient algorithms that depend on the number of terms rather than the degree \cite{Iliman16}. These methods for polynomial optimization rely on the equivalence between global optimization of a function and verification of its positivity. Following this view, Chandrasekaran et al. proposed the seminal Sums-of-AM/GM Exponential (SAGE) certificate of signomial positivity \cite{ChandrasekaranShah16}. The SAGE certificate is based on decomposing signomials into sums of parts, each of which is positive and efficiently verifiable. The key insight is that if a signomial has at most one negative term, certifying its positivity may be reduced to checking the feasibility of a small set of convex constraints. Murray et al. generalized SAGE certificates to signomial positivity over a convex set, called Conditional SAGE, derived using convex duality \cite{MurrayChandrasekaranWierman21}. The generalization shows that when a signomial has at most one negative term, certifying its positivity over a convex set may again be reduced to checking the feasibility of a small set of convex constraints. The same authors have also adapted the SAGE certificate to polynomials and have shown certain equivalences between SAGE and SONC \cite{MurrayChandrasekaranWierman21}. \\\\
Certification of function positivity is studied due to its computational equivalence to optimization. Common to all of this work is the notion of a \textit{hierarchy}. SOS, SONC, SAGE certificates all provide only sufficient conditions for positivity; not all positive polynomials or signomials may be certified directly via SOS, SONC or SAGE. However, with additional computational complexity, a larger subset of positive functions may be verified as such. The specific way in which computational complexity is increased in the hierarchy depends on individual results. In many results, additional computational complexity is imposed by multiplication of an extra function to the one in question or inclusion of additional terms in the decomposition. The hierarchy is \textit{complete} if all positive functions of a given class may be verified as such at some finite level of the hierarchy. This article provides a completeness result for the Conditional SAGE hierarchy. The result guarantees that after multiplication of a function with sufficiently many terms, every signomial that is positive over a compact convex set may be verified via the Conditional SAGE certificate.

\section{Related Work and Contribution}
In algebraic geometry, Positivstellensatz results characterize positive polynomials. Positivstellensatz results come in different forms. The earliest results, due to Artin and P{\'o}lya in the 1920s, characterize the global positivity of polynomials; Reznick later proved results of the same type. Given a polynomial $p(\xv)$ that is globally positive, these theorems show that there exist a specific function $w(\xv)$ and $p \in \mathbb{Z}_+$ such that $w(\xv)^p p(\xv)$ can be easily certified to be positive—either by having positive coefficients or by being a sum of squares (SOS). \\

By contrast, Positivstellensatz results due to Stengle, Putinar, and Schm{\"u}dgen from the late 20th century provide certificates for a polynomial $p(\xv)$ that is positive over a semialgebraic set $\mathcal{K} = \{\xv: g_i(\xv) \geq 0, \; i = 1 \ldots m \}$ by expressing $p(\xv)$ as a composition of $(g_i(\xv))_{i = 1 \ldots m}$ and some SOS polynomials. More recently, Dressler et al. have shown a Schm{\"u}dgen-type Positivstellensatz using SONC polynomials \cite{Dressler17}.  \\

Positivstellensatz results have also been extended to signomials. In 2008, Delzell extended P{\'o}lya's result to signomials with rational exponents and showed that it is impossible to extend it to signomials with arbitrary exponents \cite{Delzell2008}. In introducing SAGE, Chandrasekaran et al. showed that for any globally positive signomial $f(\xv)$ with rational exponents satisfying certain assumptions, there exist a specific function $w(\xv)$ and $p \in \mathbb{Z}_+$ such that $w(\xv)^p f(\xv)$ is a SAGE signomial \cite{ChandrasekaranShah16}. This may be viewed as an analogue of a Reznick-type Positivstellensatz for signomials \cite{Reznick1995}. The same article also presented a convergent hierarchy for signomial positivity over a constrained set defined by signomial inequalities. In particular, the result shows that if a signomial is positive on a compact set defined by signomial inequalities, the Lagrangian relaxation may be verified via the SAGE certificate. \\

The contribution of this article is a completeness theorem showing that for any signomial function $f(\xv)$ with rational exponents that is positive over a compact convex set $\mathcal{X}$, there exist $p \in \mathbb{Z}_+$ and a specific positive definite function $w(\xv)$ such that $w(\xv)^p f(\xv)$ may be verified by the Conditional SAGE certificate induced by $\mathcal{X}$. The current result has a few notable characteristics. First, unlike the convergent hierarchy for unconstrained SAGE due to Chandrasekaran et al., it does not require assumptions on the exponents beyond rationality. Second, the convergent hierarchy holds for an \textit{arbitrary} compact convex set. Thus, it is distinct from three lines of earlier work. Classical Positivstellensatz results consider polynomials that are positive either globally or on semialgebraic sets, which are defined by finitely many polynomial inequalities. Delzell's extension of P{\'o}lya-type Positivstellensatz to signomials assumes positivity over the nonnegative orthant. The constrained SAGE hierarchy due to Chandrasekaran et al. assumes that the constrained set is defined by signomial inequalities. \\

In the proofs of previous constrained SAGE and recent SONC hierarchies, redundant constraints arising from the compactness assumption are added to guarantee the Archimedean property in the set generated by the functions defining the constrained set, which allows one to appeal to results from representation theorems. The proof of the current result also uses redundant constraints, but instead of appealing to representation theorems, it reduces the problem to a particular Positivstellensatz result for polynomials over a semialgebraic set.

\section{Background and Notations}

\subsection{Notations}
We use bold fonts to denote vectors and matrices. Use $v_i$ to denote the $i$th coordinate of the vector $\vv$ and $\vv_{\backslash i} \in \mathbb{R}^{n-1}$ to denote the vector $\vv \in \mathbb{R}^n$ with the $i$th coordinate removed. Given a matrix $\Av \in \mathbb{R}^{n \times m}$, use $\Av_i \in \mathbb{R}^m$ to denote its $i$th row, and use $\Av_{\backslash i} \in \mathbb{R}^{(n-1) \times m}$ to denote the submatrix with the $i$th row removed. Given a row vector $\uv \in \mathbb{R}^n$ and a column vector $\vv \in \mathbb{R}^n$, use $\uv \vv$ to denote the dot product between them. Given a function $f: \mathbb{R} \to \mathbb{R}$, overload its definition by allowing a vector input for which the output is the elementwise map of the entries of the vector. For example, given $\xv \in \mathbb{R}^n$, $\exp(\xv) = [\exp(x_1), \ldots, \exp(x_n)]^\top$. Use $\mathbb{R}[\xv]$ to denote the ring of multivariate polynomials on $\mathbb{R}^n$ with real coefficients. For a polynomial $f \in \mathbb{R}[\xv]$, use $\text{deg}(f)$ to denote its degree. \\\\
Given $\cv \in \mathbb{R}^\ell$ and $\Av \in \mathbb{R}^{\ell \times n}$, define:
\begin{align*}
    \text{Sig}(\cv, \Av)(\xv) = \sum_{j=1}^\ell c_j \exp (\Av_j \xv).
\end{align*}

\subsection{Certificate of Positivity and Optimization}

Certification of function positivity is motivated by its equivalence to optimization. Consider the following constrained optimization problem:
\begin{align*}
    f^* = \inf_{\xv \in \mathcal{X}} f(\xv),
\end{align*}
which may be reduced to checking the positivity of a function over the constrained set:
\begin{align*}
    f^* = \sup \lambda \; \mbox{ such that } f(\xv) - \lambda \geq 0 \; \forall \xv \in \mathcal{X}.
\end{align*}

Consider a function class $F$, and a subset $F_\mathcal{X} \subseteq F$ such that $\inf_{\xv \in \mathcal{X}} f(\xv) \geq 0 \iff f \in F_\mathcal{X}$. The optimization problem is solved if membership in $F_\mathcal{X}$ can be checked:
\begin{align*}
    f^* = \sup \lambda \; \mbox{ such that } f(\xv) - \lambda \in F_\mathcal{X}.
\end{align*}

Naturally, certifying function positivity is no easier than the optimization problem itself. Of interest is to find a sufficient condition for function positivity; that is, we aim to develop a tractable set $G \subset F_\mathcal{X}$ such that checking whether $\text{Sig}(\cv, \Av) \in G$ is sufficiently easy. Then:
\begin{align*}
    f_G = \sup \lambda \; \mbox{ such that } f(\xv) - \lambda \in G
\end{align*}
is a tractable problem and $f_G \leq f^*$.
\subsection{SAGE}
Such efficiently verifiable approximations of positivity have been developed for signomials over convex sets. In this section, we describe the Conditional SAGE certificate for signomial positivity. Although the SAGE certificate was originally developed as a certificate for \textit{global} positivity, Murray et al. have generalized SAGE to positivity over an arbitrary convex set $\mathcal{X} \subset \mathbb{R}^n$, which encompasses the case when $\mathcal{X} = \mathbb{R}^n$. The authors of this article have developed this generalization independently from Murray et al. using similar techniques. Below, we present key definitions and theorems for Conditional SAGE.

\begin{definition}(Conditional AGE Signomials)
    $\text{AGE}(\Av, \mathcal{X}, i)$ is the cone of signomials with at most one negative term at the $i$th index that are nonnegative over $\mathcal{X}$. Given $\Av \in \mathbb{R}^{\ell \times n}$ and $\mathcal{X} \subset \mathbb{R}^n$, the $i$th AGE cone with respect to $\Av$ and $\mathcal{X}$ is:
    \begin{align*}
        \text{AGE}(\Av, \mathcal{X}, i) = \{\text{Sig}(\cv, \Av): \cv_{\backslash i} \geq 0 \mbox{ and } \text{Sig}(\cv, \Av)(\xv) \geq 0 \; \forall \xv \in \mathcal{X} \}.
    \end{align*}
\end{definition}

By definition, $\text{AGE}(\Av, \mathcal{X}, i)$ is a cone; it is easy to verify that it is closed under addition and nonnegative scaling. Fixing a given set of exponent vectors $\Av$, a convex set $\mathcal{X}$, and an index $i$, we may also define the set of coefficients for which the resulting signomial is a conditional AGE signomial.

\begin{definition}(Conditional AGE Cone)
    $C_{AGE}(\Av, \mathcal{X}, i)$ is the cone of coefficients with at most one negative term at the $i$th index such that the resulting signomial is nonnegative over $\mathcal{X}$. Given $\Av \in \mathbb{R}^{\ell \times n}$ and $\mathcal{X} \subset \mathbb{R}^n$, the $i$th AGE coefficient cone with respect to $\mathcal{X}$ and $\Av$ is:
    \begin{align*}
        C_{AGE}(\Av, \mathcal{X}, i) = \{\cv \in \mathbb{R}^\ell: \cv_{\backslash i} \geq 0 \mbox{ and } \text{Sig}(\cv, \Av)(\xv) \geq 0 \; \forall \xv \in \mathcal{X} \}.
    \end{align*}
\end{definition}

Equipped with the above definitions, we may define Conditional SAGE signomials.

\begin{definition}(Conditional SAGE Signomials)
      $\text{SAGE}(\Av, \mathcal{X})$ is the Minkowski sum of $\text{AGE}(\Av, \mathcal{X}, i)$ for $i = 1 \ldots \ell$. 
    \begin{align*}
        \text{SAGE}(\Av, \mathcal{X}) &= \sum_{i=1}^\ell \text{AGE}(\Av, \mathcal{X}, i). 
    \end{align*}
    Fixing a given set of exponent vectors $\Av$ and a convex set $\mathcal{X}$, a Conditional SAGE signomial in $\text{SAGE}(\Av, \mathcal{X})$ is one with exponentials defined by $\Av$ that can be decomposed into parts, each of which is nonnegative over $\mathcal{X}$ and has at most one negative term, occurring at different indices.
\end{definition}

Again, fixing a given set of exponent vectors $\Av$ and convex set $\mathcal{X}$, we may also define a set of coefficients for which the resulting signomials are Conditional SAGE signomials.

\begin{definition}(Conditional SAGE Cone)
    $C_{\text{SAGE}}(\Av, \mathcal{X})$ is the Minkowski sum of $C_{AGE}(\Av, \mathcal{X}, i)$ for $i = 1 \ldots \ell$. 
    \begin{align*}
        C_{\text{SAGE}}(\Av, \mathcal{X}) = \sum_{i=1}^\ell C_{AGE}(\Av, \mathcal{X}, i).
    \end{align*}
\end{definition}

Let $C_{NNS}(\Av, \mathcal{X}) = \{\cv \in \mathbb{R}^\ell : \text{Sig}(\cv, \Av)(\xv) \geq 0 \; \forall \xv \in \mathcal{X}\}$ denote the cone of coefficients for which the resulting signomial is nonnegative over $\mathcal{X}$. By definition, $C_{\text{SAGE}}(\Av, \mathcal{X}) \subseteq C_{NNS}(\Av, \mathcal{X})$. 

The following theorem shows that $C_{AGE}(\Av, \mathcal{X}, i)$ is a tractable set via convex constraints.

\begin{theorem}\cite{MurrayChandrasekaranWierman21}
\label{thm:cage_cone}
Given $\Av \in \mathbb{R}^{\ell \times n}$, $\mathcal{X} \subset \mathbb{R}^n$, and $i \in [\ell]$. Let $\sigma_\mathcal{X}(\lambdav) = \sup_{\xv \in \mathcal{X}} \lambdav^\top \xv$, and let $D(\vv, \cv) = \sum_j v_j \log(v_j / c_j)$ denote the relative entropy function. Then:
\begin{align*}
    C_{AGE}(\Av, \mathcal{X}, i) &= \{ \cv \in \mathbb{R}^\ell:  \exists \vv \in \mathbb{R}^{\ell-1} \mbox{ and } \lambdav \in \mathbb{R}^n \mbox{ such that } \\
    & \quad\quad \sigma_\mathcal{X}(\lambdav) + D(\vv, \cv_{\backslash i})  - \vv^\top \mathbf{1} \leq c_i  \\ 
    & \quad\quad [\Av_{\backslash i} - \mathbf{1}\Av_i]^\top \vv + \lambdav = \mathbf{0} \mbox{ and } \cv_{\backslash i} \geq \mathbf{0} \}.
\end{align*}
\end{theorem}

\begin{proof}
Let $\delta_\mathcal{X}(\cdot)$ denote the indicator function on the set $\mathcal{X}$. A vector $\cv \in \mathbb{R}^\ell$ is in the cone $C_{AGE}(\Av, \mathcal{X}, i)$ if and only if:
\begin{align*}
    \sum_{j=1}^\ell c_j \exp (\Av_j\xv) \geq 0 \; \forall \xv \in \mathcal{X} \mbox{ and } \cv_{\backslash i} \geq 0 \\
    \iff \sum_{j \neq i} c_j \exp (\Av_j \xv) \geq - c_i \exp (\Av_i \xv) \; \forall \xv \in \mathcal{X} \mbox{ and } \cv_{\backslash i} \geq 0 \\
    \iff \sum_{j \neq i} c_j \exp ((\Av_j - \Av_i) \xv) \geq - c_i \; \forall \xv \in \mathcal{X} \mbox{ and } \cv_{\backslash i} \geq 0 \\
    \iff p^* = \inf_{\xv \in \mathbb{R}^n} \delta_\mathcal{X}(\xv) + \sum_{j \neq i} c_j \exp ((\Av_j - \Av_i) \xv) \geq - c_i.
\end{align*}

We may apply Fenchel duality to this minimization problem. The resulting dual is:
\begin{align*}
    d^* = \sup_{\substack{\lambdav \in \mathbb{R}^n \\ \vv \in \mathbb{R}^{\ell-1} \\ (\Av_{\backslash i} - \mathbf{1}\Av_i)^\top \vv + \lambdav = \mathbf{0}}} - \sigma_\mathcal{X}(\lambdav) - D(\vv,\cv_{\backslash i}) + \vv^\top \mathbf{1}.
\end{align*}

When $\mathcal{X}$ is nonempty, strong duality holds by Theorem 31.1 of \cite{Rockafellar70}, since the relative interior of $\mathcal{X}$ is nonempty and the second function is finite everywhere. Consequently:
\begin{align*}
    p^* \geq -c_i \iff -d^* \leq c_i,
\end{align*}
and we obtain the desired result. When $\mathcal{X}$ is empty, $p^* = +\infty$ by definition of the indicator function. By choosing $\vv = \mathbf{0}$ and $\lambdav = \mathbf{0}$, we have $d^* = \infty$, and the result also follows.
\end{proof}

As a corollary, $C_{\text{SAGE}}$ is characterized by the following set of convex constraints:

\begin{corollary}
    Given $\Av \in \mathbb{R}^{\ell \times n}$ and $\mathcal{X} \subset \mathbb{R}^n$, we have:
\begin{align*}
    C_{\text{SAGE}}(\Av, \mathcal{X}) &= \{ \cv \in \mathbb{R}^\ell:  \exists \cv^{(i)} \in \mathbb{R}^{\ell}, \vv^{(i)} \in \mathbb{R}^{\ell-1}, \lambdav^{(i)} \in \mathbb{R}^n \mbox{ such that } \\
    & \quad\quad \sum_{j=1}^\ell \cv^{(j)} = \cv \;, \sigma_\mathcal{X}(\lambdav^{(i)}) + D(\vv^{(i)}, \cv^{(i)}_{\backslash i})  - \vv^{(i) \top} \mathbf{1} \leq c^{(i)}_i, \\ 
    & \quad\quad [\Av_{\backslash i} - \mathbf{1}\Av_i]^\top \vv^{(i)} + \lambdav^{(i)} = \mathbf{0}, \; \cv^{(i)}_{\backslash i} \geq \mathbf{0} \; \forall i \}.
\end{align*}
\end{corollary}

There are $O(\ell n)$ constraints defined by $O(\ell (\ell + n))$ variables. The complexity of the problem does grow with both the dimension of the input to the signomial and the number of terms, but it remains highly tractable.
\subsubsection{SAGE Hierarchy}
$C_{\text{SAGE}}(\Av, \mathcal{X})$ is an inner approximation of the $C_{NNS}(\Av, \mathcal{X})$ cone. We may, however, find larger and more accurate inner approximations through what is known as \textit{modulation} in the literature. Define the hierarchy of SAGE cones as: 

\begin{definition}
    Given $\Av \in \mathbb{R}^{\ell \times n}$ and $\mathcal{X} \subset \mathbb{R}^n$, the $p$th-level SAGE cone with respect to $\Av$ and $\mathcal{X}$ is:
    \begin{align*}
                C_{\text{SAGE}}^{(p)}(\Av, \mathcal{X}) &= \{\cv \in \mathbb{R}^\ell: (\exp(\Av \xv)^\top \mathbf{1})^p \text{Sig}(\cv, \Av) \in \text{SAGE}(E_{p+1}(\Av), \mathcal{X}) \}.
    \end{align*}
    Here, $E_{p+1}(\Av)$ denotes the matrix whose rows are all sums $\Av_{j_1} + \ldots + \Av_{j_{p+1}}$ with $j_1, \ldots, j_{p+1} \in [\ell]$.
\end{definition}

That is, instead of certifying a given signomial to be positive, we verify whether its product with a positive definite function is positive. Multiplication by a positive definite function does not change positivity, and thus the test is valid. Equivalently, expanding the above definition and paying attention to the coefficients of the larger signomial after multiplication by the extra function gives:
\begin{align*}
    C_{\text{SAGE}}^{(p)}(\Av, \mathcal{X}) &= \{ \cv \in \mathbb{R}^\ell: \exists \cv^{(i)} \in \mathbb{R}^{\ell^{(p+1)}}, \vv^{(i)} \in \mathbb{R}^{\ell^{(p+1)}-1}, \Tilde{\Av} \in \mathbb{R}^{\ell^{(p+1)} \times n}, \text{ and } \lambdav^{(i)} \in \mathbb{R}^n  \\
    & \quad\quad \text{such that } \Tilde{\Av}_{j_1 + \ell \cdot j_2 + \ldots + \ell^p \cdot j_{p+1}} = \Av_{j_1} + \ldots + \Av_{j_{p+1}} \text{ for } j_1 \ldots j_{p+1} = 1 \ldots \ell, \\
& \quad\quad c^{(i)}_{k + 0 \cdot \ell} = c^{(i)}_{k + 1 \cdot \ell} = \ldots = c^{(i)}_{k + \ell^{(p+1)}-\ell } \; \forall k = 1 \ldots \ell, \;\; \sum_{j=1}^{\ell^{(p+1)}} \cv^{(j)}_{1:\ell} = \cv, \\
& \quad\quad \sigma_\mathcal{X}(\lambdav^{(i)}) + D(\vv^{(i)}, \cv^{(i)}_{\backslash i})  - \vv^{(i) \top} \mathbf{1} \leq c^{(i)}_i,  \\ 
& \quad\quad [\Tilde{\Av}_{\backslash i} - \mathbf{1}\Tilde{\Av}_i]^\top \vv^{(i)} + \lambdav^{(i)} = \mathbf{0}, \text{ and } \cv^{(i)}_{\backslash i} \geq \mathbf{0} \; \text{ for } i = 1 \ldots \ell^{(p+1)} \}.
\end{align*}

A more careful construction can reduce some of the constraints. In particular, the equality constraints on $\cv^{(i)}$ are simply used to construct a large vector with repeated subvectors, and the constraints on $\Tilde{\Av}$ simply indicate that its rows are linear combinations of rows of $\Av$. In implementation, we may avoid such constraints by reusing variables. Ignoring these constraints, at the $p$th level of the hierarchy, there are $O(\ell^{(p+1)} n)$ constraints defined by $O(\ell^{(p+1)} (\ell^{(p+1)} + n))$ variables. Upon setting $p = 0$, we recover $ C_{\text{SAGE}}(\Av, \mathcal{X})$. Moreover, we have the following relations. 

\begin{theorem}
\label{thm:sage_p_inequality}
    Given $\Av \in \mathbb{R}^{\ell \times n}$, $\mathcal{X} \subset \mathbb{R}^n$, and any $p \in \mathbb{Z}_{+}$, 
    \begin{align*}
        C_{\text{SAGE}}^{(p)}(\Av, \mathcal{X}) \subseteq C_{\text{SAGE}}^{(p+1)}(\Av, \mathcal{X}) \subseteq C_{NNS}(\Av, \mathcal{X}).
    \end{align*}
\end{theorem}

\begin{proof}
Let $\cv \in C_{\text{SAGE}}^{(p)}(\Av, \mathcal{X})$ and let $F(\xv) = (\exp(\Av \xv)^\top \mathbf{1})^p \text{Sig}(\cv, \Av)(\xv)$. By definition, $F \in \text{SAGE}(E_{p+1}(\Av), \mathcal{X})$, so $F = \sum_i F_i$, where each $F_i$ is nonnegative over $\mathcal{X}$ and has at most one negative term. Then:
\begin{align*}
    (\exp(\Av \xv)^\top \mathbf{1})^{p+1} \text{Sig}(\cv, \Av)(\xv)
    &= \sum_{j=1}^\ell \exp(\Av_j \xv) F(\xv) \\
    &= \sum_{j=1}^\ell \sum_i \exp(\Av_j \xv) F_i(\xv).
\end{align*}
Since $\exp(\Av_j \xv) > 0$, each $\exp(\Av_j \xv) F_i(\xv)$ is nonnegative over $\mathcal{X}$ and has at most one negative term. Its exponents are $\Av_j$ plus exponents in $E_{p+1}(\Av)$, so they lie in $E_{p+2}(\Av)$. Thus each $\exp(\Av_j \xv) F_i(\xv)$ is a Conditional AGE signomial with exponents $E_{p+2}(\Av)$. Since $\text{SAGE}(E_{p+2}(\Av), \mathcal{X})$ is closed under addition, $(\exp(\Av \xv)^\top \mathbf{1})^{p+1} \text{Sig}(\cv, \Av) \in \text{SAGE}(E_{p+2}(\Av), \mathcal{X})$, and hence $\cv \in C_{\text{SAGE}}^{(p+1)}(\Av, \mathcal{X})$.

Next, let $\cv \in C_{\text{SAGE}}^{(p+1)}(\Av, \mathcal{X})$. By definition, $(\exp(\Av \xv)^\top \mathbf{1})^{p+1} \text{Sig}(\cv, \Av) \in \text{SAGE}(E_{p+2}(\Av), \mathcal{X})$, so it is nonnegative over $\mathcal{X}$. Since $\exp(\Av \xv)^\top \mathbf{1} > 0$:
\begin{align*}
    (\exp(\Av \xv)^\top \mathbf{1})^{p+1} \text{Sig}(\cv, \Av)(\xv) \geq 0 \; \forall \xv \in \mathcal{X}
    &\iff \text{Sig}(\cv, \Av)(\xv) \geq 0 \; \forall \xv \in \mathcal{X} \\
    &\implies \cv \in C_{NNS}(\Av, \mathcal{X}).
\end{align*}
\end{proof}

The above shows that the hierarchy of SAGE cones provides an increasingly accurate inner approximation of nonnegative signomials.

\subsubsection{A Toy Example}
Below is a toy example that illustrates \cref{thm:sage_p_inequality} concretely. Let $\cv = [1, -1, -1]^\top$ and $\Av = [\av_1, \av_2, \av_3]^\top$ where $\av_1 = [1.5, 0]^\top, \av_2 = [0, 1]^\top, \av_3 = [0, -1]^\top$. Consider the signomial $\text{Sig}(\cv, \Av): \mathbb{R}^2 \to \mathbb{R}$:
\begin{align*}
    \text{Sig}(\cv, \Av)(\xv) = \exp(1.5x_1) - \exp(x_2) - \exp(-x_2),
\end{align*}
and the convex set:
\begin{align*}
    \mathcal{X} = \{\xv: x_1 = 1, \; -1 \leq x_2 \leq 1 \}.
\end{align*}

One may verify that $\text{Sig}(\cv, \Av)(\xv) > 0 \; \forall \xv \in \mathcal{X}$ (hint: the two negative terms are minimized at different locations in $\mathcal{X}$). However, $\text{Sig}(\cv, \Av) \notin \text{SAGE}(\Av, \mathcal{X}) \iff \cv \notin C_{\text{SAGE}}(\Av, \mathcal{X})$. To see this, recall that a SAGE signomial must have a decomposition under which each summand has at most one negative term and is nonnegative over $\mathcal{X}$. Since there are two negative terms in $\text{Sig}(\cv, \Av)$, the only possible decomposition satisfying such conditions consists of two summands, each including one of the two negative terms $-\exp(x_2)$ and $-\exp(-x_2)$. The two terms both attain $-e$ over $\mathcal{X}$, but the single positive term $\exp(1.5x_1)$ of $\text{Sig}(\cv, \Av)$ is constant over $\mathcal{X}$ at $e^{1.5} < 2e$. Thus, it cannot be distributed across the two negative terms such that both summands are nonnegative over $\mathcal{X}$. \\\\

On the other hand:
\begin{align*}
    (\exp(\Av \xv)^\top \mathbf{1}) \text{Sig}(\cv, \Av)(\xv) &= (\exp(1.5x_1) + \exp(x_2) + \exp(-x_2)) (\exp(1.5x_1) - \exp(x_2) - \exp(-x_2)) \\
                                                &= \exp(3x_1) - \exp(2x_2) - \exp(-2x_2) - 2 \\
                                                &= (0.4\exp(3x_1) - \exp(2x_2)) + (0.4\exp(3x_1) - \exp(-2x_2)) + (0.2\exp(3x_1) - 2)
\end{align*}

One may check that all three terms are positive over $\mathcal{X}$. Thus, $(\exp(\Av \xv)^\top \mathbf{1}) \text{Sig}(\cv, \Av)(\xv) \in \text{SAGE}(E_2(\Av), \mathcal{X}) \iff \cv \in C^{(1)}_{\text{SAGE}}(\Av, \mathcal{X})$. Thus, $C_{\text{SAGE}}(\Av, \mathcal{X}) \subset C^{(1)}_{\text{SAGE}}(\Av, \mathcal{X})$, where the inclusion is strict. In this toy example, the technique of modulation allows certification of a signomial that is positive over a constrained set where the direct SAGE method would otherwise fail.

\subsubsection{SAGE Relaxation}
Given the hierarchy of SAGE cones, we may formulate a hierarchy of relaxations for signomial optimization. Consider a signomial $\text{Sig}(\cv, \Av)$ and a convex set $\mathcal{X}$, and let $f^* = \inf_{\xv \in \mathcal{X}} \text{Sig}(\cv, \Av)(\xv)$. The relaxation may be formulated as follows:
\begin{align*}
    f^{(p)}_{\text{SAGE}} &= \sup_{\lambda} \lambda \mbox{ s.t. } \text{Sig}(\cv, \Av) - \lambda \in \text{SAGE}^{(p)}(\Av, \mathcal{X}) \\
    &= \sup_{\lambda} \lambda \mbox{ s.t. } \cv - \lambda \mathbf{e}_k \in C_{\text{SAGE}}^{(p)}(\Av, \mathcal{X}),
\end{align*}
where $\Av_k = \mathbf{0}$ is a zero row of $\Av$, appended with coefficient $c_k = 0$ if $\Av$ has none, so that subtracting the constant $\lambda$ changes only $c_k$, and $\mathbf{e}_k$ is the $k$th standard basis vector.

\Cref{thm:sage_p_inequality} tells us that the Conditional SAGE cones are increasingly tight inner approximations of the set of signomials that are nonnegative over a given convex set. As a result, we have $f^{(p)}_{\text{SAGE}} \leq f^{(p+1)}_{\text{SAGE}} \leq f^*$.

\subsection{Overview of Positivstellensatz}
In this subsection we highlight some classical Positivstellensatz results to contrast with the main result of the article. \\\\
The first Positivstellensatz is due to Artin in 1927, in response to Hilbert's 17th problem. 
\begin{theorem}(Artin’s Positivstellensatz \cite{Artin27})
Consider a polynomial $p(\xv)$. If $p(\xv)$ is globally positive, then there exists a nonzero SOS polynomial $q(\xv)$ such that $q(\xv)p(\xv)$ is a SOS polynomial. 
\end{theorem}
Verification of a polynomial as sum of squares, as in the above theorem, may be reformulated as an SDP, as studied extensively by Lasserre and Parrilo \cite{Parrilo00, Lasserre01}. The theorem is the basis for certifying positivity of a polynomial, and a hierarchy may be induced by limiting the degree of $q(\xv)$ to search over. \\\\
The following result due to P{\'o}lya in 1928 provides an \textit{optimization-free} Positivstellensatz. That is, verification of positivity is easy at a sufficiently high level of the hierarchy.  
\begin{theorem}(P{\'o}lya's Positivstellensatz \cite{Polya28})
Consider a homogeneous polynomial $f(\xv)$. If $f(\xv) > 0$ for all $\xv \in \mathbb{R}^n_+ \backslash \{\mathbf{0}\}$, then there exists some $p \in \mathbb{Z}_+$ such that $(\sum_{i=1}^n x_i)^p f(\xv)$ is a polynomial with only positive coefficients. 
\end{theorem}
The above two results concern polynomials that are positive globally or on the nonnegative orthant. The following result from 1991 characterizes polynomials that are positive over semialgebraic sets. 
\begin{theorem}(Schm{\"u}dgen's Positivstellensatz \cite{Schmudgen91})
\label{schmudgen_Positivstellensatz}
Consider a polynomial $p(\xv)$ and compact semialgebraic set $\mathcal{K} = \{\xv \in \mathbb{R}^n : g_i(\xv) \geq 0 \;\; i = 1 \ldots m \}$ defined by polynomials $(g_i(\xv))_{i = 1 \ldots m}$. If $p(\xv) > 0 \; \forall \xv \in \mathcal{K}$, then:
\begin{align*}
    p(\xv) = s_0(\xv) + \sum_{i \in [m]} s_i(\xv)g_i(\xv) + \sum_{i_1,i_2 \in [m]^2} s_{i_1 i_2}(\xv)g_{i_1}(\xv)g_{i_2}(\xv) + \ldots \\ + \sum_{i_1, \ldots i_m \in [m]^m} s_{i_1 \ldots i_m}(\xv)g_{i_1}(\xv)g_{i_2}(\xv) \ldots g_{i_m}(\xv),
\end{align*}
\end{theorem}
where $s_0(\xv), s_i(\xv), s_{i_1i_2}(\xv), \ldots, s_{i_1 \ldots i_m}(\xv)$ are SOS polynomials. The above is essentially a search problem over SOS polynomials, and may again be converted to a SDP by representing SOS polynomials by semidefinite matrices. The problem may also be relaxed by restricting the number of compositions of $(g_i(\xv))_{i = 1 \ldots m}$ (i.e. truncating later summation terms on the RHS). The relaxation yields a converging hierarchy for the problem of verifying polynomial positivity over a semialgebraic set.  \\\\
Much more recently, Dickinson and Povh developed a new Positivstellensatz for polynomials positive over a semialgebraic set, which is closely related to P{\'o}lya's Positivstellensatz. 
\begin{theorem}(Dickinson's Positivstellensatz \cite{DickinsonPovh15})
\label{dickinson_positivstellensatz}
Consider some homogeneous polynomials $f_0, \ldots, f_m \in \mathbb{R}[\xv]$. If $f_0(\xv) > 0 \; \forall \xv \in \mathbb{R}_+^n \cap \bigcap_{i=1}^m f^{-1}_i(\mathbb{R}_+) \backslash \{ \mathbf{0} \}$, then there exist some $p \in \mathbb{Z}_+$ and some homogeneous polynomials with nonnegative coefficients $g_0, \ldots, g_m \in \mathbb{R}[\xv]$ such that $(\xv^\top \mathbf{1})^p f_0(\xv) = g_0(\xv) + \sum_{i=1}^m f_i(\xv)g_i (\xv)$. 
\end{theorem}
Readers are encouraged to visit the original article for a full exposition of the result \cite{DickinsonPovh15}. The significance of the above result in contrast to Schm{\"u}dgen's is that by a multiplication of an additional function, the polynomial in question may be represented by a single sum of products of two polynomials, one defining the constrained set and another one with positive coefficients, without increasing the number of polynomials to be composed in the products. Homogeneity is assumed in both the polynomial in question and the ones defining the constrained set. We will crucially rely on the above result to prove the main result of this article.

\section{The Main Result}
In this section we describe the main result of the article. 
\begin{theorem}
\label{thm:main_result}
Let $\cv \in \mathbb{R}^{\ell}$ and $\Av = [\Av_1 \ldots \Av_\ell]^\top \in \mathbb{Q}^{\ell \times n}$. Consider the signomial $\text{Sig}(\cv, \Av)$ and a compact convex set $\mathcal{X} \subset \mathbb{R}^n $. If $\text{Sig}(\cv, \Av)(\xv) > 0 \; \forall \xv \in \mathcal{X}$, then there exists some $p \in \mathbb{Z}_+$ such that $(\exp(\Av \xv)^\top \mathbf{1})^p \text{Sig}(\cv, \Av)$ is a Conditional SAGE signomial in $\text{SAGE}(E_{p+1}(\Av), \mathcal{X})$. 
\end{theorem}
The theorem may be thought of as a Positivstellensatz for signomials over an arbitrary compact convex set. The rationality of the exponent vectors is required. 

\begin{corollary}
Under the assumptions of Theorem~\ref{thm:main_result}, except that $\text{Sig}(\cv, \Av)$ need not be positive on $\mathcal{X}$, and with $\mathcal{X}$ nonempty, let $f^{(p)}_{\text{SAGE}} = \sup\{\lambda : \text{Sig}(\cv, \Av) - \lambda \in \text{SAGE}^{(p)}(\Av, \mathcal{X})\}$ as in Section~3.3.3.
Then $f^{(p)}_{\text{SAGE}} \to \inf_{\xv \in \mathcal{X}} \text{Sig}(\cv, \Av)(\xv)$
as $p \to \infty$.
\end{corollary}

\begin{proof}
Let $f^* = \inf_{\xv \in \mathcal{X}} \text{Sig}(\cv, \Av)(\xv)$. For any $\epsilon > 0$,
$\text{Sig}(\cv, \Av) - (f^* - \epsilon) \geq \epsilon > 0$ on $\mathcal{X}$, so by
Theorem~\ref{thm:main_result} it is certified at some level $p_\epsilon$, that is,
$\text{Sig}(\cv, \Av) - (f^* - \epsilon) \in \text{SAGE}^{(p_\epsilon)}(\Av, \mathcal{X})$.
Then $f^{(p_\epsilon)}_{\text{SAGE}} \geq f^* - \epsilon$. By
Theorem~\ref{thm:sage_p_inequality}, the bounds are nondecreasing in $p$ and at most $f^*$,
so $f^* - \epsilon \leq f^{(p)}_{\text{SAGE}} \leq f^*$ for all $p \geq p_\epsilon$.
Since $\epsilon$ is arbitrary, $f^{(p)}_{\text{SAGE}} \to f^*$.
\end{proof}

\section{Proof of the Main Result}
\label{sec:proof}
The proof is structured as follows. We first note that any compact convex set $\mathcal{X}$ may be expressed as an intersection of a possibly infinite set of rational halfspaces $H_{\mathcal{X}}$. We then apply the change of variables $\yv = \exp(\Av \xv)$ and show that positivity of the signomial $\text{Sig}(\cv, \Av)$ over the intersection of the halfspaces $H_{\mathcal{X}}$ implies the positivity of a corresponding polynomial $p(\yv)$ over a set $\mathcal{T}^{(\Av, \mathcal{X})}_{++}$, the intersection of the positive orthant and a set defined by possibly infinitely many polynomial inequalities. We make modifications to this set to show that the positivity of $p(\yv)$ over such a set implies its positivity over $\mathcal{T}_+^{(\Av, \mathcal{X})*}$, the intersection of the nonnegative orthant and a homogeneous semialgebraic set defined by finitely many homogeneous polynomials. We then appeal to Dickinson's Positivstellensatz to claim that $(\yv^\top \mathbf{1})^p p(\yv)$, for some $p \in \mathbb{Z}_+$, can be expressed as a composition of polynomials with only positive coefficients and the homogeneous polynomials defining $\mathcal{T}_+^{(\Av, \mathcal{X})*}$. Undoing the change of variables, we show through this composition that the original signomial, after multiplication by an extra function, $(\exp(\Av \xv)^\top \mathbf{1})^p \text{Sig}(\cv, \Av)$, is a Conditional SAGE signomial. \\

Without loss of generality, we may make the following assumptions on the exponent vectors $\Av \in \mathbb{Q}^{\ell \times n}$ defining $\text{Sig}(\cv, \Av)$: \\
\begin{enumerate}[label=(\alph*)]
    \item The first $n$ rows $(\Av_j)_{j=1 \ldots n}$ are linearly independent.
    \item $\Av_{n+1} = \mathbf{0}$. \\
\end{enumerate} 

These assumptions are not restrictive. To satisfy the first condition, we may select a set of linearly independent exponents as the first $n$. The proof is easily generalized to the case where the span of the exponent vectors has dimension less than $n$. The second condition is not restrictive either: for any $k$, $\text{Sig}(\cv, \Av)(\xv) = \exp(\Av_k \xv)\, \text{Sig}(\cv, \Av - \mathbf{1}\Av_k)(\xv)$, where $\Av - \mathbf{1}\Av_k$ has a zero row, and multiplying by the positive term $\exp(\Av_k \xv)$ only shifts exponents, so the result for $\Av - \mathbf{1}\Av_k$ implies the result for $\Av$. It is used to homogenize certain constraints in the proof. \\ 

At its core, the proof is a reduction to Dickinson's Positivstellensatz. The proof is divided into sections explaining each step in the reduction.

\subsection{Representation of Compact Convex Set as Rational Halfspaces}
In this subsection, we discuss the following result, which connects a compact convex set and rational halfspaces. It has been shown by de Carli Silva and Tunçel, and we paraphrase it below. 

\begin{theorem}(\cite{Silva20})
\label{thm:convex_set_rational}
    A rational halfspace $h \subset \mathbb{R}^n$ is a set $\{\xv \in \mathbb{R}^n : \wv^\top \xv \leq b \}$ where $\wv \in \mathbb{Q}^n$ and $b \in \mathbb{Q}$. Given a compact convex set $\mathcal{X}$, there exists a set of rational halfspaces $H$ such that $\mathcal{X} = \cap_{h \in H} h$.
\end{theorem}

\begin{proof}
The proof is nonconstructive. It suffices to show that for each $\yv \in \mathbb{R}^n \backslash \mathcal{X}$, there exist $\wv \in \mathbb{Q}^n$ and $b \in \mathbb{Q}$ such that $\wv^\top \xv \leq b < \wv^\top \yv \; \forall \xv \in \mathcal{X}$. Recall that $\mathbb{Q}$ is dense in $\mathbb{R}$. Thus, it suffices to show that for each $\yv \in \mathbb{R}^n \backslash \mathcal{X}$, there exists $\wv \in \mathbb{Q}^n$ such that $\sigma_{\mathcal{X}}(\wv) < \wv^\top \yv$. Then, by density, we may find $b \in \mathbb{Q}$ such that $\sigma_{\mathcal{X}}(\wv) < b < \wv^\top \yv$, which implies the desired result. \\

Let $\Xi$ be the intersection of all halfspaces defined by rational vectors containing $\mathcal{X}$. That is, $\Xi = \{\xv: \wv^\top \xv \leq \sigma_{\mathcal{X}}(\wv), \; \forall \wv \in \mathbb{Q}^n \}$. The proof is complete if $\mathcal{X} = \Xi$, since by definition of $\Xi$, if $\yv \in \mathbb{R}^n \backslash \Xi$, there exists some $\wv \in \mathbb{Q}^n$ such that $\sigma_{\mathcal{X}}(\wv) < \wv^\top \yv $. Indeed, $\mathcal{X} = \Xi$, as shown below. \\

First, suppose $\bar{\xv} \in \mathcal{X}$. For any $\wv \in \mathbb{Q}^n$, $\wv^\top \bar{\xv} \leq \sigma_{\mathcal{X}}(\wv)$ by definition, and thus $\bar{\xv} \in \Xi$. Hence, $\mathcal{X} \subseteq \Xi$ by construction. Next, suppose $\bar{\xv} \in \Xi$. Consider any $\bar{\wv} \in \mathbb{R}^n$. Let $(\wv_k)_{k \in \mathbb{N}} \subset \mathbb{Q}^n$ be a sequence of vectors converging to $\bar{\wv}$. Since $\mathcal{X}$ is compact, the support function $\sigma_{\mathcal{X}}(\cdot)$ is continuous. The image of a convergent sequence under a continuous map converges, and thus $(\sigma_{\mathcal{X}}(\wv_k))_{k \in \mathbb{N}}$ converges to $\sigma_{\mathcal{X}}(\bar{\wv})$. By assumption and definition of $\Xi$, $\wv^\top_k \bar{\xv} \leq \sigma_{\mathcal{X}}(\wv_k)$ for all $k$. Taking the limit preserves the inequality, so $\bar{\wv}^\top \bar{\xv} \leq \sigma_{\mathcal{X}}(\bar{\wv})$. This implies $\bar{\xv} \in \{\xv \in \mathbb{R}^n: \wv^\top \xv \leq \sigma_{\mathcal{X}}(\wv) \; \forall \wv \in \mathbb{R}^n \}$. By Theorem 13.1 of \cite{Rockafellar70}, a closed convex set is the intersection of the closed halfspaces containing it, i.e., $\mathcal{X} = \{\xv \in \mathbb{R}^n: \wv^\top \xv \leq \sigma_{\mathcal{X}}(\wv) \; \forall \wv \in \mathbb{R}^n \}$, and thus $\bar{\xv} \in \mathcal{X}$. 
\end{proof} 

Without loss of generality, for the compact convex set $\mathcal{X}$, let $H_{\mathcal{X}}$ denote the set of rational halfspaces defining $\mathcal{X}$. We have shown:
\begin{align*}
    \mathcal{X} = \cap_{h \in H_{\mathcal{X}}} h.
\end{align*}

\subsection{Signomial to Polynomial}
\label{sec:variable_change}
We have now reduced the positivity of the signomial $\text{Sig}(\cv, \Av)$ over the compact convex set $\mathcal{X}$ to its positivity over the intersection of a set of rational halfspaces $H_\mathcal{X}$. In this section, we apply the change of variables $\yv = \exp(\Av \xv) \in \mathbb{R}^\ell_{++}$. We obtain a polynomial after the change of variables, namely $p(\yv) = \cv^\top \yv$. Further, we are interested in the necessary and sufficient conditions on $\yv = \exp(\Av \xv) \in \mathbb{R}^\ell_{++}$ given $\xv \in \cap_{h \in H_{\mathcal{X}}} h$.

\subsubsection{Restriction of Exponentials as Polynomial Constraints}
When $\Av \in \mathbb{Q}^{\ell \times n}$ has rank less than $\ell$, its range does not cover $\mathbb{R}^\ell$, which adds restrictions to $\yv = \exp(\Av \xv) \in \mathbb{R}^\ell_{++}$. Since the first $n$ exponents $(\Av_j)_{j=1 \ldots n}$ are linearly independent, the rest of the vectors may be expressed as rational linear combinations of the first $n$ vectors, and thus they are constrained by the first $n$ vectors. For $\Av_j$ with $j \geq n+2$, $\Av_j = \sum_{i=1}^{n+1} w_i^{(j)} \Av_i$, where $\wv^{(j)} \in \mathbb{Q}^{n+1}$ for all $j$. Then:
\begin{align*}
    y_j &= \exp (\Av_j \xv) = \exp \left(\sum_{i=1}^{n+1} w^{(j)}_i \Av_i \xv\right) = \prod_{i=1}^{n+1} \exp (\Av_i\xv )^{w^{(j)}_i} = \prod_{i=1}^{n+1} y_i ^{w^{(j)}_i} = \prod_{i=1}^{n} y_i ^{w^{(j)}_i}.
\end{align*}
The last step follows from the fact that $\Av_{n+1} = \mathbf{0}$. Since the $w^{(j)}_i$ are rational, we may raise both sides to the power of a common denominator to clear the fractions. For example, $y_j = y_1^{\frac{1}{2}} y_2^{\frac{1}{4}} \iff y_j^4 =  y_1^2 y_2 $. \\

We may apply this operation to $y_j$ for all $j \geq n+2$. The operation is only valid for $\yv$ in the positive orthant.
\begin{align*}
    y_j =  \exp (\Av_j \xv) = \prod_{i=1}^{n} y_i ^{w^{(j)}_i} \; \forall j=n+2 \ldots \ell 
    & \iff y_j^{\lambda^{(j)}_j} =  \prod_{i=1}^{n} y_i ^{\lambda^{(j)}_i} \; \forall j=n+2 \ldots \ell \\
    & \iff y_j^{\lambda^{(j)}_j} \prod_{i \; : \; \lambda^{(j)}_i < 0} y_i^{-\lambda^{(j)}_i} - \prod_{i \; : \; \lambda^{(j)}_i > 0} y_i^{\lambda^{(j)}_i} = 0 \; \forall j=n+2 \ldots \ell.
\end{align*}
where the $\lambdav^{(j)}$ are obtained from the above procedure.

\subsubsection{Rational Halfspace Constraint to Polynomial Constraint}
To characterize the constraint $\xv \in \cap_{h \in H_{\mathcal{X}}} h$ on $\yv = \exp(\Av \xv)$, we begin by considering a single rational halfspace constraint on $\xv \in \mathbb{R}^n$. Let $h = \{\xv: \wv^\top \xv \leq d \} \in H_{\mathcal{X}}$, where $\wv \in \mathbb{Q}^n$ and $d \in \mathbb{Q}$. Recall that $\Av \in \mathbb{Q}^{\ell \times n}$. Given such a constraint on $\xv$, what can we say about the exponential of the rational linear map of $\xv$, $\yv = \exp(\Av \xv)$? \\

First, consider the rational linear map $\Av \xv$ subject to a rational halfspace constraint on $\xv$: $\wv^\top \xv \leq d$. We may find a rational halfspace constraint on the linear map. By assumption, we have $\text{rank}(\Av) = n$, which implies that the columns of $\Av$ are linearly independent. Thus, there exists a left inverse $\Lv = (\Av^\top \Av)^{-1}\Av^\top \in \mathbb{Q}^{n \times \ell}$ such that $\Lv\Av \xv = \xv$. The rationality of $\Lv$ follows from the rationality of $\Av$ and the fact that the inverse operation preserves rationality. Letting $\bv = \wv^\top \Lv \in \mathbb{Q}^\ell$, we have $\bv^\top (\Av \xv) \leq d$. \\\\

We then consider the exponential of the linear map $\exp(\Av \xv)$ subject to a rational halfspace constraint on $\xv$. Beginning with $\yv = \exp(\Av \xv) \iff \log \yv = \Av \xv$, a series of algebraic operations follows below. In doing so, we assume $\yv \in \mathbb{R}^\ell_{++}$. 
\begin{align*}
    \bv(\Av \xv) \leq d
    & \iff \bv(\log \yv) \leq d \iff \sum_{i=1}^\ell b_i (\log y_i) \leq d \\
    & \iff  \sum_{i=1}^\ell  (\log y_i^{b_i}) \leq d  \iff  \log \left(\prod_{i=1}^\ell y_i^{b_i}\right) \leq d \\
    & \iff \prod_{i=1}^\ell y_i^{b_i} \leq \exp(d)
    \iff  \prod_{i \; : \; b_i > 0} y_i^{b_i} \leq \exp (d) \prod_{i \; : \; b_i < 0} y_i^{-b_i}.
\end{align*}
The last step multiplies both sides to clear negative exponents. For example, $y_1^{2}y_2^{-3} \leq 1 \iff y_1^2 \leq y_2^{3}$. Now, since $\bv \in \mathbb{Q}^\ell$ has rational entries, we may raise both sides to the power of a common denominator to clear fractions. Let $m$ be the common denominator of $(b_i)_{i = 1 \ldots \ell}$:
\begin{align*}
 \prod_{i \; : \; b_i > 0} y_i^{b_i} \leq \exp (d) \prod_{i \; : \; b_i < 0} y_i^{-b_i} \\
 \iff \prod_{i \; : \; b_i > 0} y_i^{m b_i} \leq  \exp (m d ) \prod_{i \; : \; b_i < 0} y_i^{-m b_i} \\
 \iff  \exp ( m d ) \prod_{i \; : \; b_i < 0} y_i^{-m b_i}  - \prod_{i \; : \; b_i > 0} y_i^{m b_i} \geq 0.
\end{align*}
This is indeed a polynomial inequality. 

\subsubsection{Intersection of Rational Halfspaces as an Intersection of Polynomial Constraints} 
In the above, a single rational halfspace constraint has been shown to be equivalent to a polynomial constraint. To extend the above to the intersection of possibly infinitely many halfspaces, we simply take the intersection of the polynomial inequalities generated from them. \\

Let $K_{\mathcal{X}}$ be a possibly infinite set of indices corresponding to the rational halfspaces defining $\mathcal{X}$. For each $k \in K_{\mathcal{X}}$, let $\bv^{(k)}$, $d^{(k)}$, and $m^{(k)}$ be the $\bv$, $d$, and $m$ above for the $k$th halfspace, and let $\gammav^{(k)} = m^{(k)} \bv^{(k)} \in \mathbb{Z}^\ell$ and $c^{(k)} = \exp (m^{(k)} d^{(k)}) \in \mathbb{R}_{++}$. We may write:
\begin{align*}
c^{(k)} \prod_{i \; : \; \gamma^{(k)}_i < 0} y_i^{-\gamma^{(k)}_i} - \prod_{i \; : \; \gamma^{(k)}_i > 0} y_i^{\gamma^{(k)}_i} \geq 0 \;\; \forall k \in K_{\mathcal{X}}.
\end{align*} 

Now, consider the following set that depends on $\Av \in \mathbb{Q}^{\ell \times n}$ and $\mathcal{X} \subset \mathbb{R}^n$. Recall the assumption $\Av_{n+1} = \mathbf{0}$, which implies $y_{n+1} = \exp(\Av_{n+1} \xv) = 1$. 
\begin{align*}
\mathcal{T}_{++}^{(\Av, \mathcal{X})} &=  \{\yv = \exp (\Av \xv): \xv \in \mathcal{X} \} \\
&= \{ \yv \in \mathbb{R}_{++}^\ell: \;  y_{n+1} = 1, \; y_j^{\lambda^{(j)}_j} \prod_{i \; : \; \lambda^{(j)}_i < 0} y_i^{-\lambda^{(j)}_i} - \prod_{i \; : \; \lambda^{(j)}_i > 0} y_i^{\lambda^{(j)}_i} = 0 \;\; \forall j=n+2 \ldots \ell, \\  
& \quad\quad\quad c^{(k)} \prod_{i \; : \; \gamma^{(k)}_i < 0} y_i^{-\gamma^{(k)}_i} - \prod_{i \; : \; \gamma^{(k)}_i > 0} y_i^{\gamma^{(k)}_i} \geq 0 \; \forall k \in K_{\mathcal{X}} \} \\
&=  \{ \yv \in \mathbb{R}_{++}^\ell: \; y_{n+1} = 1, \; p_1^{(j)}(\yv) - p_2^{(j)}(\yv) = 0 \; \forall j = n+2 \ldots \ell, \\ 
& \quad\quad\quad q_1^{(k)}(\yv) - q_2^{(k)}(\yv) \geq 0 \; \forall k \in K_{\mathcal{X}} \},
\end{align*}
where in the last expression we have written the terms abstractly. Each of $p_1^{(j)}(\yv)$, $p_2^{(j)}(\yv)$, $q_1^{(k)}(\yv)$, and $q_2^{(k)}(\yv)$ is a monomial. By construction, we have $\xv \in \mathcal{X} \iff \yv \in \mathcal{T}_{++}^{(\Av, \mathcal{X})}$. \\

The change of variables $\yv = \exp(\Av \xv)$ subject to $\xv \in \mathcal{X}$ reduces signomial positivity over a compact convex set to polynomial positivity over the intersection of the positive orthant and a set defined by possibly infinitely many polynomial inequalities. In summary, we have shown the following:
\begin{align*}
\text{Sig}(\cv, \Av)(\xv) > 0 \; \forall \xv \in \mathcal{X} \implies \cv^\top \yv > 0 \; \forall \yv \in \mathcal{T}_{++}^{(\Av, \mathcal{X})}.
\end{align*}

\subsection{Polynomial Positivity to Positivstellensatz}
\label{sec:posivity_to_Positivstellensatz}
We recall, however, that Dickinson's Positivstellensatz assumes polynomial positivity over the intersection of the nonnegative orthant and a set defined by finitely many homogeneous polynomial inequalities, excluding the origin. Namely, there are three conditions that $\mathcal{T}_{++}^{(\Av, \mathcal{X}) }$ does not satisfy:
\begin{enumerate}
    \item $\mathcal{T}_{++}^{(\Av, \mathcal{X}) }$ is the intersection of the positive orthant and a set of polynomial inequalities.
    \item $\mathcal{T}_{++}^{(\Av, \mathcal{X}) }$ is defined by polynomials that are possibly nonhomogeneous.
    \item $\mathcal{T}_{++}^{(\Av, \mathcal{X}) }$ is defined by possibly infinitely many polynomials. \\
\end{enumerate}  

The goal of this section is to describe modifications to the set $\mathcal{T}_{++}^{(\Av, \mathcal{X}) }$ such that the resulting set satisfies the premises of Dickinson's Positivstellensatz, while the positivity of $\cv^\top \yv$ is preserved on the modified set. That is, there exists a set $\mathcal{T}_+^{(\Av, \mathcal{X})*}$, the intersection of the nonnegative orthant and a semialgebraic set defined by homogeneous polynomials, such that:
\begin{align*}
     \cv^\top \yv > 0 \; \forall \yv \in \mathcal{T}_{++}^{(\Av, \mathcal{X})} \implies \cv^\top \yv > 0 \; \forall \yv \in \mathcal{T}_+^{(\Av, \mathcal{X})*} \backslash \{ \mathbf{0} \}.
\end{align*}

\subsubsection{Intersection with the Nonnegative Orthant}
In this subsection, we make modifications to $\mathcal{T}_{++}^{(\Av, \mathcal{X})}$ so that the resulting set is the intersection of the nonnegative orthant and a set of polynomial inequalities. First, consider the following set that extends $\mathcal{T}_{++}^{(\Av, \mathcal{X})}$ to the nonnegative orthant. 
\begin{align*}
\mathcal{T}_{+}^{(\Av, \mathcal{X})} &= \{ \yv \in \mathbb{R}_{+}^\ell: \; y_{n+1} = 1, \; p_1^{(j)}(\yv) - p_2^{(j)}(\yv) = 0 \; \forall j = n+2 \ldots \ell, \\ 
&\quad\quad\quad q_1^{(k)}(\yv) - q_2^{(k)}(\yv) \geq 0 \; \forall k \in K_{\mathcal{X}} \}.
\end{align*}

Surprisingly, after this modification, the positivity of $\cv^\top \yv$ need not be preserved on the extended set.
\begin{proposition}
\label{prop:nonegative_orthant_counter_example}
    There exist $\Av \in \mathbb{Q}^{\ell \times n}$, $\cv \in \mathbb{R}^{\ell}$, and a compact convex set $\mathcal{X}$ such that:
    \begin{align*}
        \cv^\top \yv > 0 \; \forall \yv \in \mathcal{T}_{++}^{(\Av, \mathcal{X})} \not\implies \cv^\top \yv > 0 \; \forall \yv \in \mathcal{T}_{+}^{{(\Av, \mathcal{X})}} \backslash \{ \mathbf{0} \}.
    \end{align*}
\end{proposition}

\begin{proof}
We prove this by example. Let $\Av = [\av_1, \av_2, \av_3]^\top$, where $\av_1 = [1, 0]^\top$, $\av_2 = [0, 1]^\top$, $\av_3 = [0, 0]^\top$, $\cv = [1,1, -1]^\top$, and $\mathcal{X} = \{\xv \in \mathbb{R}^2: -2x_1 + x_2 \leq 0, \; x_1 - 2x_2 \leq 0, \; x_1 + x_2 \leq 1 \}$. $\mathcal{X}$ is conveniently defined to be a compact intersection of halfspaces. \\\\

Let $\yv = \exp (\Av \xv)$. We have:
\begin{align*}
    y_1 = \exp (x_1), \; y_2 = \exp(x_2), \; y_3 = 1.
\end{align*}
Also:
\begin{align*}
    \xv \in \mathcal{X} &\iff -2x_1 + x_2 \leq 0,  x_1 - 2x_2 \leq 0, \text{ and } x_1 + x_2 \leq 1 \\
    &\iff -2 \log y_1 + \log y_2 \leq 0, \log y_1 - 2 \log y_2 \leq 0, \text{ and } \log y_1 + \log y_2 \leq 1 \\
    &\iff \log (y_1^{-2}y_2) \leq 0,  \log (y_1 y_2^{-2}) \leq 0,  \text{ and } \log(y_1y_2) \leq 1 \\
    &\iff y_1^2 - y_2 \geq 0, y_2^2 - y_1  \geq 0, \text{ and } e - y_1y_2 \geq 0.
\end{align*}
Therefore, $\mathcal{T}_{++}^{(\Av, \mathcal{X}) }  = \{ \yv \in \mathbb{R}_{++}^3: y_3 = 1, \; y_1^2 - y_2 \geq 0, \; y_2^2 - y_1 \geq 0, \; e - y_1y_2 \geq 0 \}$ and $\mathcal{T}_{+}^{(\Av, \mathcal{X}) }  = \{ \yv \in \mathbb{R}_{+}^3: \; y_3 = 1, \; y_1^2 - y_2 \geq 0, \; y_2^2 - y_1 \geq 0, \; e - y_1y_2 \geq 0 \}$. \\\\

We may check that $\cv^\top \yv \geq 1 \; \forall \yv \in \mathcal{T}_{++}^{(\Av, \mathcal{X})}$, as the constraints $y_1^2 - y_2 \geq 0$ and $y_2^2 - y_1 \geq 0$ can only be satisfied when $y_1, y_2 \geq 1$ given $\yv > \mathbf{0}$. However, letting $\bar{\yv} = [0, 0, 1]^\top \in \mathcal{T}_+^{{(\Av, \mathcal{X})}}$, we have $\cv^\top \bar{\yv} = -1 < 0$. 
\end{proof}

While $\yv = \exp(\Av \xv)$ implies $\yv \in \mathbb{R}^\ell_{++}$, the resulting polynomial inequalities after the algebraic operations do not. The explicit restriction of $\mathcal{T}_{++}^{(\Av, \mathcal{X})}$ to $\mathbb{R}^\ell_{++}$ retains the constraint, but extending the definition to the nonnegative orthant includes ``jumps'' where some coordinates of $\yv$ are zero and satisfy the polynomial inequalities. \\

We may avoid such jumps by adding redundant constraints due to the compactness of $\mathcal{X}$. Since $\mathcal{X}$ is compact, $\exp(\Av (\mathcal{X}))$ is compact, as it is the image of $\mathcal{X}$ under a continuous map. It is also contained in $\mathbb{R}^\ell_{++}$. Thus, for each $i$, there exist bounds $0 < l_i \leq u_i$ such that $l_i \leq y_i \leq u_i$ for all $\yv \in \exp(\Av (\mathcal{X}))$. Define the following set:
\begin{align*}
    \Upsilon = \{\yv: u_i - y_i \geq 0, \; y_i - l_i \geq 0 \quad \forall i  = 1 \ldots \ell \}.
\end{align*}
$\Upsilon$ consists of polynomial inequalities and is contained in the positive orthant since $\yv \notin \Upsilon$ if $y_i = 0$ for any $i$. By definition of $l_i$ and $u_i$, $\Upsilon \supseteq \mathcal{T}_{++}^{(\Av, \mathcal{X})}$. Thus, $\mathcal{T}_{+}^{(\Av, \mathcal{X}) } \cap \Upsilon = \mathcal{T}_{++}^{(\Av, \mathcal{X})}$, where $\mathcal{T}_{+}^{(\Av, \mathcal{X})} \cap \Upsilon$ is the intersection of the nonnegative orthant and polynomial inequalities. In other words, while $\mathcal{T}_{++}^{(\Av, \mathcal{X})}$ was the intersection of the positive orthant and a set of polynomial inequalities, we have written an equivalent set $\mathcal{T}_{+}^{(\Av, \mathcal{X}) } \cap \Upsilon$ as the intersection of the nonnegative orthant and a set of polynomial inequalities.

\subsubsection{Homogeneous Polynomial Inequalities}
\label{sec:positivity_over_homogeneous polynomials}
In this subsection, we show a modification to the polynomials in the definition of $\mathcal{T}_+^{(\Av, \mathcal{X})} \cap \Upsilon$ so that the resulting polynomials are homogeneous. We then show that the positivity of $\cv^\top \yv$ is preserved on the modified set. 
\begin{align*}
\mathcal{T}_+^{(\Av, \mathcal{X})} \cap \Upsilon &=  \{ \yv \in \mathbb{R}_{+}^\ell: \; y_{n+1} = 1, \; p_1^{(j)}(\yv) - p_2^{(j)}(\yv) = 0 \; \forall j = n+2 \ldots \ell, \\ 
& \quad\quad q_1^{(k)}(\yv) - q_2^{(k)}(\yv) \geq 0 \; \forall k \in K_{\mathcal{X}}, \\ 
& \quad\quad u_i - y_i \geq 0, \; y_i - l_i \geq 0 \; \forall i  = 1 \ldots \ell \}.
\end{align*}
We modify the polynomials to be homogeneous by making the following transformation. Let $[x]_+ = \max(x, 0)$.
\begin{align*}
\widetilde{\mathcal{T}}_+^{(\Av, \mathcal{X})} \cap \widetilde{\Upsilon} &=  \{ \yv \in \mathbb{R}_{+}^\ell: \\ 
& \quad\quad\quad\quad y_{n+1}^{[\text{deg}(p_2^{(j)}) - \text{deg}(p_1^{(j)})]_+}p_1^{(j)}(\yv) - y_{n+1}^{[\text{deg}(p_1^{(j)}) - \text{deg}(p_2^{(j)})]_+} p_2^{(j)}(\yv) = 0 \; \forall j = n+2 \ldots \ell, \\
& \quad\quad\quad\quad y_{n+1}^{[\text{deg}(q_2^{(k)}) - \text{deg}(q_1^{(k)})]_+} q_1^{(k)}(\yv) - y_{n+1}^{[\text{deg}(q_1^{(k)}) - \text{deg}(q_2^{(k)})]_+} q_2^{(k)}(\yv) \geq 0 \; \forall k \in K_{\mathcal{X}}, \\
& \quad\quad\quad\quad u_i y_{n+1} - y_i \geq 0, \; y_i - l_i y_{n+1} \geq 0 \; \forall i  = 1 \ldots \ell \}.
\end{align*}
We have multiplied the appropriate terms by powers of $y_{n+1}$ so that the resulting polynomials are homogeneous. We have also removed the condition $y_{n+1} = 1$, as it is not a homogeneous polynomial and cannot be modified as such. The set has been modified considerably. We claim the following:

\begin{theorem}
\label{thm:homogeneous_polynomials}
Consider a signomial $\text{Sig}(\cv, \Av)$ and compact convex set $\mathcal{X}$ as in \cref{thm:main_result}. Then:
\begin{align*}
    \cv^\top\yv > 0 \;\; \forall \yv \in (\mathcal{T}_+^{(\Av, \mathcal{X})}\cap \Upsilon)  \implies \cv^\top\yv > 0 \;\; \forall \yv \in ({\widetilde{\mathcal{T}}_+^{(\Av, \mathcal{X})}} \cap \widetilde{\Upsilon}) \backslash \{ \mathbf{0} \}.
\end{align*}
\end{theorem}

\begin{proof}
Consider some $\yv \in (\widetilde{\mathcal{T}}_+^{(\Av, \mathcal{X})} \cap \widetilde{\Upsilon}) \backslash \{ \mathbf{0} \}$. Then $y_{n+1} \neq 0$, since otherwise $\yv = \mathbf{0}$ by the constraints $u_i y_{n+1}  - y_i \geq 0 \; \forall i$. Let $\widetilde{\yv} = \yv / y_{n+1}$. Since $\widetilde{\mathcal{T}}_+^{(\Av, \mathcal{X})} \cap \widetilde{\Upsilon}$ is a semialgebraic set defined by homogeneous polynomials, it is closed under positive scaling, and thus $\widetilde{\yv} \in \widetilde{\mathcal{T}}_+^{(\Av, \mathcal{X})} \cap \widetilde{\Upsilon}$. Since $\widetilde{\yv}_{n+1} = 1$, the conditions for $\widetilde{\mathcal{T}}_+^{(\Av, \mathcal{X})} \cap \widetilde{\Upsilon}$ reduce to the conditions for $\mathcal{T}_+^{(\Av, \mathcal{X})} \cap \Upsilon$, and thus $\widetilde{\yv} \in \mathcal{T}_+^{(\Av, \mathcal{X})} \cap \Upsilon$. By assumption, $\cv^\top\widetilde{\yv} > 0$. Since $y_{n+1} > 0$, $\cv^\top\yv = \cv^\top(y_{n+1} \widetilde{\yv}) > 0$.
\end{proof}

\subsubsection{Infinite to Finite Polynomial Inequalities}
\label{sec:finite_polynomials}
${\widetilde{\mathcal{T}}_+^{(\Av, \mathcal{X})}} \cap \widetilde{\Upsilon}$ is defined by possibly infinitely many polynomial inequalities. In this subsection, we show that the positivity of $\cv^\top \yv$ over such a set implies its positivity over a set defined by finitely many polynomial inequalities. \\

Consider the following theorem, extended from the original statement in \cite{DickinsonPovh15}.
\begin{theorem}
\label{thm:infinite_to_finite}
    Consider a set of homogeneous polynomials $\{f_0\} \cup \{ f_i \; | \; i \in I \} \subseteq \mathbb{R}[\xv]$ with possibly infinite cardinality. If $f_0(\xv) > 0$ for all $\xv \in \mathbb{R}^n_{+} \cap \bigcap_{i \in I} f_i^{-1}(\mathbb{R}_+) \backslash \{\mathbf{0}\}$, then there exists a subset $J \subseteq I$ of finite cardinality such that $f_0(\xv) > 0$ for all $\xv \in \mathbb{R}^n_{+} \cap \bigcap_{i \in J} f_i^{-1}(\mathbb{R}_+) \backslash \{\mathbf{0}\}$.
\end{theorem}

The proof of the above theorem is deferred to the appendix and is adapted from \cite{DickinsonPovh15}. In other words, if a homogeneous polynomial is positive over the intersection of infinitely many homogeneous polynomial inequalities, it is positive over the intersection of some finite subset of such inequalities. \\

Writing each equality constraint $f(\yv) = 0$ as the two inequalities $f(\yv) \geq 0$ and $-f(\yv) \geq 0$, in \cref{thm:infinite_to_finite}, let $\mathbb{R}_+^n \cap \bigcap_{i \in I} f^{-1}_i(\mathbb{R}_+) = \widetilde{\mathcal{T}}_+^{(\Av, \mathcal{X})} \cap \widetilde{\Upsilon}$ and $\mathbb{R}_+^n \cap \bigcap_{i \in J} f_i^{-1}(\mathbb{R}_+) = \mathcal{T}_+^{(\Av, \mathcal{X})*}$. Then we have $\cv^\top\yv > 0 \; \forall \yv \in (\widetilde{\mathcal{T}}_+^{(\Av, \mathcal{X})} \cap \widetilde{\Upsilon}) \backslash \{ \mathbf{0} \} \implies  \cv^\top\yv > 0 \; \forall \yv \in \mathcal{T}_+^{(\Av, \mathcal{X})*} \backslash \{ \mathbf{0} \}$. Although the proof is nonconstructive, the finitely many polynomials defining $\mathcal{T}_+^{(\Av, \mathcal{X})*}$ are a subset of those defining $\widetilde{\mathcal{T}}_+^{(\Av, \mathcal{X})} \cap \widetilde{\Upsilon}$. \\

With the above, we have completed the reduction as below:
\begin{align*}
     \text{Sig}(\cv, \Av)(\xv) > 0 \; \forall \xv \in \mathcal{X}  &\implies \cv^\top \yv > 0 \; \forall \yv \in {\mathcal{T}}_{++}^{(\Av, \mathcal{X})} \\
    &\implies \cv^\top \yv > 0 \; \forall \yv \in \mathcal{T}_+^{(\Av, \mathcal{X})} \cap \Upsilon \\
    &\implies \cv^\top \yv > 0 \; \forall \yv \in (\widetilde{\mathcal{T}}_+^{(\Av, \mathcal{X})} \cap \widetilde{\Upsilon}) \backslash \{ \mathbf{0} \} \\
    &\implies \cv^\top\yv > 0 \; \forall \yv \in \mathcal{T}_+^{(\Av, \mathcal{X})*} \backslash \{ \mathbf{0} \},
\end{align*} 
where $\mathcal{T}_+^{(\Av, \mathcal{X})*}$ is the intersection of the nonnegative orthant and a semialgebraic set defined by finitely many homogeneous polynomials, as desired. In addition, we have the following relation. 

\begin{theorem}
\label{thm:change_of_variable_equivalence}
Consider any $\Av \in \mathbb{Q}^{\ell \times n}$ and $\mathcal{X}$ as in \cref{thm:main_result}. With the change of variables $\yv = \exp(\Av \xv)$:
\begin{align*}
    \xv \in \mathcal{X} \implies \yv \in \mathcal{T}_+^{(\Av, \mathcal{X})*} \backslash \{ \mathbf{0} \}.
\end{align*}
\end{theorem}

\begin{proof}
We have already shown $\xv \in \mathcal{X} \iff \yv \in \mathcal{T}_{++}^{(\Av, \mathcal{X})} = \mathcal{T}_{+}^{(\Av, \mathcal{X})} \cap \Upsilon$. We have $y_{n+1} = \exp (\Av_{n+1} \xv) = 1$ by assumption on $\Av$. By construction of $\widetilde{\mathcal{T}}_+^{(\Av, \mathcal{X})} \cap \widetilde{\Upsilon}$, if $\yv \in \mathcal{T}_{+}^{(\Av, \mathcal{X})} \cap \Upsilon$ with $y_{n+1} = 1$, then $\yv \in \widetilde{\mathcal{T}}_+^{(\Av, \mathcal{X})} \cap \widetilde{\Upsilon}$, since the modification is null when $y_{n+1} = 1$. Also, $\widetilde{\mathcal{T}}_+^{(\Av, \mathcal{X})} \cap \widetilde{\Upsilon} \subseteq \mathcal{T}_+^{(\Av, \mathcal{X})*}$ since the latter is defined by a subset of the polynomials defining the former. Therefore, $\yv \in \mathcal{T}_+^{(\Av, \mathcal{X})*}$. Lastly, $\yv = \exp(\Av \xv) \neq \mathbf{0}$, which completes the proof.
\end{proof}

\subsection{Positivstellensatz to Conditional SAGE}
In this subsection, we appeal to Dickinson's Positivstellensatz to find a representation of the polynomial $p(\yv)$. Observe that all homogeneous polynomials defining $\mathcal{T}_+^{(\Av, \mathcal{X})*}$ are of the form $m_1(\yv) - m_2(\yv)$, the difference of two monomials. Without loss of generality, we may write:
\begin{align*}
    \mathcal{T}_+^{(\Av, \mathcal{X})*} = \{\yv \in \mathbb{R}^\ell_+: m_1^{(j)}(\yv) - m_2^{(j)}(\yv) \geq 0, \; j=1 \ldots m\}.
\end{align*}
As a consequence of Dickinson's Positivstellensatz (\cref{dickinson_positivstellensatz}), for some $p \in \mathbb{Z}_+$, there exist homogeneous polynomials with nonnegative coefficients $(g_i (\yv))_{i=0 \ldots m}$ such that:
\begin{align*}
    (\exp(\Av \xv)^\top \mathbf{1})^p \sum_{j=1}^{\ell} c_j \exp (\Av_j \xv) &= (\yv^\top \mathbf{1})^p \cv^\top\yv \\
    &= g_0(\yv) + \sum_{j=1}^m g_j (\yv) (m_1^{(j)}(\yv) - m_2^{(j)}(\yv)).
\end{align*} 

We show explicitly below that the last expression is a summation of functions, each of which is nonnegative over $\mathcal{X}$ and has at most one negative term. By undoing the change of variables $\yv = \exp(\Av \xv)$, this then implies that the first expression is a SAGE signomial. \\

Dickinson's Positivstellensatz does not bound the degrees of the polynomials $(g_i (\yv))_{i=0 \ldots m}$. A key observation is that $(\yv^\top \mathbf{1})^p \cv^\top\yv$ is a homogeneous polynomial of degree $p+1$. Thus, it must be that $g_j (\yv) (m_1^{(j)}(\yv) - m_2^{(j)}(\yv))$ is homogeneous and $\text{deg} (g_j (\yv) (m_1^{(j)}(\yv) - m_2^{(j)}(\yv))) = p+1$ for each $j$, since we may without loss of generality ignore terms whose degrees do not equal $p+1$. Similarly, $g_0(\yv)$ is either zero or homogeneous of degree $p+1$. Moreover, for each $j$, without loss of generality, $g_j (\yv) = \sum_{k=1}^{\ell(j)} h_k^{(j)} (\yv)$, where $h_k^{(j)}$ is a monomial and $\ell(j)$ is the number of terms in the polynomial $g_j (\yv)$. Then:
\begin{align*}
    (\yv^\top \mathbf{1})^p \cv^\top\yv &= g_0(\yv) + \sum_{j=1}^m g_j (\yv) (m_1^{(j)}(\yv) - m_2^{(j)}(\yv)) \\
    &= g_0(\yv) + \sum_{j=1}^m \sum_{k=1}^{\ell(j)} h_k^{(j)} (\yv) (m_1^{(j)}(\yv) - m_2^{(j)}(\yv)) \\
    &= g_0(\exp (\Av \xv)) + \sum_{j=1}^m \sum_{k=1}^{\ell(j)} o_k^{(j)}(\xv).
\end{align*}
where $o_k^{(j)}(\xv) = h_k^{(j)} (\exp (\Av \xv)) (m_1^{(j)}(\exp (\Av \xv)) - m_2^{(j)}(\exp (\Av \xv))) = h_k^{(j)} (\yv) (m_1^{(j)}(\yv) - m_2^{(j)}(\yv))$ for all $j$ and $k$. One may verify that after undoing the change of variables, the expression is indeed a signomial in exponential form. Now:
\begin{align*}
    \xv \in \mathcal{X} &\implies \yv = \exp (\Av \xv) \in \mathcal{T}_+^{(\Av, \mathcal{X})*} \\ 
    &\implies m_1^{(j)}(\yv) - m_2^{(j)}(\yv) \geq 0 \;\; \forall j = 1 \ldots m \\
    &\implies o_k^{(j)}(\xv) = h_k^{(j)} (\exp (\Av \xv))(m_1^{(j)}(\exp (\Av \xv)) - m_2^{(j)}(\exp (\Av \xv))) \geq 0 \;\; \forall j, k,
\end{align*}
where the first implication is due to \cref{thm:change_of_variable_equivalence}. Make the following observations for all $j$ and $k$: \\
\begin{itemize}
    \item $o_k^{(j)}(\xv) \geq 0 \; \forall \xv \in \mathcal{X}$, as implied above.
    \item $o_k^{(j)}(\xv)$ has at most one negative term, namely $h_k^{(j)} (\exp \Av \xv)m_2^{(j)}(\exp \Av \xv)$. 
    \item Since $\text{deg} (g_j (\yv) (m_1^{(j)}(\yv) - m_2^{(j)}(\yv))) = p+1$, the exponents of $o_k^{(j)}(\xv)$ lie in $E_{p+1}(\Av)$. \\
\end{itemize}
These conditions imply that $o_k^{(j)}(\xv) \in \text{SAGE}(E_{p+1}(\Av), \mathcal{X})$ for all $j$ and $k$. Also, $g_0(\exp (\Av \xv))$ is a signomial with nonnegative coefficients, so $g_0(\exp (\Av \xv)) \in  \text{SAGE}(E_{p+1}(\Av), \mathcal{X})$. Since $\text{SAGE}(E_{p+1}(\Av), \mathcal{X})$ is a cone and is closed under addition, we have $(\exp(\Av \xv)^\top \mathbf{1})^p \sum_{j=1}^{\ell} c_j \exp (\Av_j \xv) = g_0(\exp (\Av \xv)) + \sum_{j=1}^m \sum_{k=1}^{\ell(j)} o_k^{(j)}(\xv) \in \text{SAGE}(E_{p+1}(\Av), \mathcal{X})$. \\

The above completes the proof of the main result (\cref{thm:main_result}). 

\section{Discussion and Future Work}
In this article, we presented a Positivstellensatz for Conditional SAGE signomials. Specifically, we established a convergent hierarchy of certificates for signomial positivity over a compact convex set. The proof takes a similar approach to recent Positivstellensatz results by reducing positivity of one form to positivity of another \cite{ChandrasekaranShah16, Ahmadi19}. In doing so, we used redundant constraints as in previous proofs, but rather than appealing directly to representation theorems, the current proof reduces to Dickinson's Positivstellensatz through algebraic operations. The redundant constraints are used in particular to avoid ``jumps'' when extending the aforementioned set to the intersection of the nonnegative orthant and polynomial inequalities, as well as to preserve positivity of the polynomial in question when the constrained set is modified to be homogeneous. The previous convergent hierarchy for unconstrained signomial programming presented by Chandrasekaran et al. required one restrictive assumption on the exponent vectors besides rationality \cite{ChandrasekaranShah16}. In the current result, that assumption may be removed by exploiting the compactness assumption in the reduction. \\

There are open questions resulting from this work. The current work does not provide a complexity estimate of the positive definite function to multiply, i.e., an upper bound on $p \in \mathbb{Z}_+$ in \cref{thm:main_result}. Previous work has studied complexity estimates for classical Positivstellensatz results \cite{Schweighofer04, Nie07}. Providing such an upper bound for the current result would be an advancement in further understanding the computational nature of constrained signomial optimization problems. It would likely be a challenge: with the current proof strategy, the complexity estimate for our result is the same as that of Dickinson's Positivstellensatz. In Dickinson's Positivstellensatz, the complexity estimate depends on the degrees of the polynomials defining the homogeneous semialgebraic set \cite{DickinsonPovh15}. In the reduction, such polynomials originate from the rational halfspace constraints defining the compact convex set $\mathcal{X}$, and the degrees of the resulting polynomials after the change of variables depend on the least common multiple of the denominators of the entries in the vectors defining the rational halfspaces -- which are unbounded. Another direction to pursue is to relax the compactness assumption on the convex set. One difficulty of such an extension lies in avoiding jumps when extending the aforementioned set to the nonnegative orthant, which likely requires redundant constraints. The authors of this article have considered this extension and have found difficulty in developing a general method to construct redundant constraints that avoid all jumps without a compactness assumption. A vastly different proof strategy would likely be necessary to obtain the above extensions.

\section{Acknowledgements}
The authors would like to thank Riley Murray for very helpful and insightful discussions.

\bibliography{ref}

\newpage

\section*{Appendix}
\textbf{Proof of \Cref{thm:infinite_to_finite}}:
Consider the set $\Omega = \{ \xv \in \mathbb{R}^n_+ : \|\xv\|_2 = 1 \}$, which is compact. $f_0(\xv) > 0$ for all $\xv \in \mathbb{R}^n_{+} \cap \bigcap_{i \in I} f_i^{-1}(\mathbb{R}_+) \backslash \{\mathbf{0}\}$ if and only if $f_0(\xv) > 0$ for all $\xv \in \Omega \cap \mathbb{R}^n_{+} \cap \bigcap_{i \in I} f_i^{-1}(\mathbb{R}_+)$, since $f_0(\xv)$ is homogeneous and its positivity is invariant to the scale of $\xv$. Thus we may restrict to the leveled set. $\Omega$ is the intersection of the level set and the nonnegative orthant since the theorem only concerns the positivity of $f_0(\xv)$ on the nonnegative orthant. Let $B = \{ \xv \in \mathbb{R}^n_+ : \|\xv\|_2 \leq 1 \}$, which is compact, convex, and contains both $\Omega$ and $\mathbf{0}$. Make the following observations:

\begin{enumerate}
    \item Without loss of generality, assume that $\deg(f_i(\xv)) \geq 1$ and $\max_{\xv \in B}\{ \| \nabla f_i(\xv) \|_2 \} \leq 1$ for all $i \in I$. For the first condition, constant constraints $f_i \equiv k \geq 0$ can be dropped, and if some $f_i \equiv k < 0$, the result holds trivially with $J = \{i\}$. The second condition may be achieved by scaling: since $\nabla f_i$ is a polynomial, it is continuous and therefore bounded on the compact set $B$, and its maximum norm on $B$ is positive since $f_i$ is nonconstant. This is not restrictive since the theorem only considers positivity of homogeneous polynomials $(f_i(\xv))_{i \in I}$, which is invariant under scaling. 
    
    \item By the mean value theorem, for any $\xv, \yv \in B$, and any $i \in I$, there exists $\alpha \in [0,1]$ such that 
    \[
    f_i(\xv) - f_i(\yv) = (\xv - \yv)^\top \nabla f_i(\alpha \xv + (1-\alpha) \yv).
    \]
    Since $B$ is convex, $\alpha \xv + (1-\alpha) \yv \in B$. Thus
    \begin{align*}
        |f_i(\xv) - f_i(\yv)| 
        &\leq \|\xv - \yv\|_2 \|\nabla f_i(\alpha \xv + (1-\alpha) \yv)\|_2 \\
        &\leq \|\xv - \yv\|_2 \max_{\zv \in B}\{ \| \nabla f_i(\zv) \|_2 \} \\
        &\leq \|\xv - \yv\|_2 \quad \forall i \in I.
    \end{align*}
    This implies that for all $i \in I$, $f_i(\xv)$ is 1-Lipschitz on $B$, with the same constant for all $i \in I$.
    
    \item Since $\deg(f_i(\xv)) \geq 1$ and $f_i(\xv)$ is homogeneous, $f_i(\mathbf{0}) = 0$. Since $\mathbf{0} \in B$, by observation (2), for all $i \in I$, 
    \[
    |f_i(\xv)| = |f_i(\xv) - f_i(\mathbf{0})| \leq \|\xv\|_2 \leq 1 \quad \forall \xv \in \Omega.
    \]
\end{enumerate}

Now define the following compact sets:
\begin{itemize}
    \item $\Omega_0 = \Omega \cap f_0^{-1}(-\mathbb{R}_+)$.
    \item $\Omega_j = \Omega \cap f_j^{-1}(\mathbb{R}_+)$.
    \item $\Omega_J = \Omega \cap \bigcap_{i \in J} f_i^{-1}(\mathbb{R}_+) = \bigcap_{i \in J} \Omega_i \quad \forall J \subseteq I$.
\end{itemize}

Observe that $f_0(\xv) > 0$ for all $\xv \in \mathbb{R}^n_{+} \cap \bigcap_{i \in I} f_i^{-1}(\mathbb{R}_+) \backslash \{\mathbf{0}\}$ if and only if $\Omega_0 \cap \Omega_I = \emptyset$. The goal is to show that there exists $J \subseteq I$ such that $\Omega_0 \cap \Omega_J = \emptyset$, which then implies $f_0(\xv) > 0$ for all $\xv \in \mathbb{R}^n_{+} \cap \bigcap_{i \in J} f_i^{-1}(\mathbb{R}_+) \backslash \{\mathbf{0}\}$. \\

Consider the function $\xi(\xv) = \sup \{ -f_i(\xv) \; | \; i \in I \}$ for all $\xv \in \Omega$. By observation (3), $\xi(\xv)$ is finite, and by observation (2), it is the supremum of 1-Lipschitz functions, so it is 1-Lipschitz and thus continuous. Since $\Omega_0 \cap \Omega_I = \emptyset$ by assumption, we have $\xi(\xv) > 0$ for all $\xv \in \Omega_0$. Moreover, by observation (3), $\xi(\xv) \in (0,1]$ for all $\xv \in \Omega_0$. Let $\epsilon = \min_{\xv \in \Omega_0} \xi(\xv)$. Since $\Omega_0$ is compact and $\xi(\xv)$ is continuous, the minimum is attained, so $\epsilon \in (0,1]$. \\

Then note the following two facts:
\begin{itemize}
    \item[(a)] For any $\xv \in \Omega_0$, there exists some $i \in I$ such that 
    \[
    -f_i(\xv) \geq \tfrac{2}{3} \xi(\xv) \geq \tfrac{2}{3} \epsilon > 0.
    \]
    \item[(b)] For any $\yv \in \Omega_0$ such that $\|\xv - \yv\|_2 \leq \tfrac{1}{3} \epsilon$, we have
    \[
    f_i(\yv) \leq f_i(\xv) + \|\xv - \yv\|_2 \leq -\tfrac{2}{3} \epsilon + \tfrac{1}{3} \epsilon < 0,
    \]
    so $\yv \notin \Omega_i$.
\end{itemize}

And now consider the algorithm below. Suppose $\zv_t$ is chosen at the $t$th iteration of the while loop. First, there exists $i \in I$ to add to $J$ by claim (a). Since $f_i(\zv_t) \leq -\tfrac{2}{3}\epsilon < 0$, by claim (b), for any $\zv_{t+1}$ such that $\|\zv_{t+1} - \zv_t\|_2 \leq \tfrac{1}{3}\epsilon$, we have $\zv_{t+1} \notin \Omega_i$, and thus $\zv_{t+1} \notin \Omega_{J \cup \{i\}}$. Therefore, in each iteration, $\zv_t$ has a distance of at least $\tfrac{1}{3}\epsilon$ from the previous ones. Since $\Omega_0$ is compact, the algorithm terminates in finite time, which implies $\Omega_0 \cap \Omega_J = \emptyset$. The desired $J \subseteq I$ is obtained. \\

\begin{algorithm}[H]
\SetAlgoLined
\caption{Finding $J \subseteq I \text{ such that } \Omega_0 \cap \Omega_J = \emptyset$}
\label{alg:iterations}
Let $J = \emptyset$ \\
\While{$\exists \zv \in \Omega_0 \cap \Omega_J$}{
    Find $i \in I$ such that $f_i(\zv) \leq -\tfrac{2}{3} \epsilon$ \\
    $J = J \cup \{i\}$
}
\Return $J$
\end{algorithm}
\end{document}